\documentclass[12pt,a4paper]{amsart}
\usepackage{etex} 
\usepackage[latin1]{inputenc}
\usepackage[T1]{fontenc}
\usepackage[right=3cm,left=3cm,top=2cm,bottom=2cm]{geometry}
\usepackage{ifthen}
\usepackage{enumerate}
\usepackage{amsmath}
\usepackage{amsthm}
\usepackage{amsfonts}
\usepackage{amssymb}
\usepackage{latexsym}
\usepackage{ae}
\usepackage{xcolor} 
\usepackage{graphics,amsmath,amssymb}
\usepackage[np]{numprint}
\usepackage{fourier} 
\usepackage{hyperref}

\npdecimalsign{\ensuremath{.}}

\title{A Note on a Unitary Analog to Redheffer's Matrix}
\author{Olivier Bordell\`{e}s}
\address{2 all\'{e}e de la combe \\ 43000 Aiguilhe \\ France}
\email{borde43@wanadoo.fr}
\date{}
\dedicatory{}

\theoremstyle{plain}
\newtheorem{theorem}{Theorem}
\newtheorem{corollary}[theorem]{Corollary}
\newtheorem{lemma}[theorem]{Lemma}
\newtheorem{prop}[theorem]{Proposition}

\theoremstyle{definition}

\theoremstyle{remark}

\newcommand{\N}{\mathbb {N}}
\newcommand{\Z}{\mathbb {Z}}

\makeatletter
\@namedef{subjclassname@2010}{%
  \textup{2010} Mathematics Subject Classification}
\makeatother

\begin{document}

\begin{abstract}
\noindent
We study a unitary analog to Redheffer's matrix. It is first proved that the determinant of this matrix is the unitary analogue to that of Redheffer's matrix. We also show that the coefficients of the characteristic polynomial may be expressed as sums of Stirling numbers of the second kind. This implies in particular that $1$ is an eigenvalue with algebraic multiplicity greater than that of Redheffer's matrix.
\end{abstract}



\maketitle

\thispagestyle{myheadings}
\font\rms=cmr8 
\font\its=cmti8 
\font\bfs=cmbx8

\section{Introduction}
\label{s1}

In 1977, Redheffer \cite{red} introduced the matrix $R_n=(r_{ij}) \in \mathcal {M}_n(\{0,1\})$ defined by
$$r_{ij} = \begin{cases} 1,& \textrm{if\ } i \mid j \, \, \mbox{or} \, \,  j=1 \\ 0, & \textrm{otherwise\ } \end{cases}$$
and has shown that
$$\det R_n = M(n) := \sum_{k=1}^{n} \mu(k)$$
where $\mu$ is the M\"{o}bius function and $M$ is the Mertens function. This determinant is clearly related to two of the most famous problems in number theory, namely the Prime Number Theorem (PNT) and the Riemann Hypothesis (RH) since it is well-known that 
$$\text{PNT} \Longleftrightarrow M(n) = o(n) \quad \text{and} \quad \text{RH} \Longleftrightarrow M(n) = O_{\varepsilon} \left ( n^{1/2 + \varepsilon} \right ).$$
These estimates remain unproven, but Vaughan \cite{vau} showed that $1$ is an eigenvalue of $R_n$ with algebraic multiplicity $n - \left \lfloor \frac{\log n}{\log 2} \right \rfloor - 1$, that $R_n$ has two "dominant" eigenvalues $\lambda_{\pm}$ such that $|\lambda_{\pm}| \asymp n^{1/2}$, and that the others eigenvalues satisfy $\lambda \ll (\log n)^{2/5}$.

The purpose of this note is to supply an analogous study to the $\{0,1\}$-matrix $R_n^* = \left( \rho_{ij} \right)$ defined by
$$\rho_{ij} = \begin{cases} 1,& \textrm{if\ } i \parallel j \, \, \mbox{or} \, \,  j=1 \\ 0, & \textrm{otherwise}. \end{cases}$$
Recall that the integer $i$ is said to be a \textit{unitary divisor} of $j$, denoted by $i  \parallel j$,  whenever
$$i \mid j \quad \textrm{and} \quad \gcd \left( i, \tfrac{j}{i} \right) = 1.$$
For instance, when $n=8$, we have
$$R_8^* = \begin{pmatrix} 1 & 1 & 1 & 1 & 1 & 1 & 1 & 1 \\ 1 & 1 & 0 & 0 & 0 & 1 & 0 & 0 \\ 1 & 0 & 1 & 0 & 0 & 1 & 0 & 0 \\ 1 & 0 & 0 & 1 & 0 & 0 & 0 & 0 \\ 1 & 0 & 0 & 0 & 1 & 0 & 0 & 0 \\ 1 & 0 & 0 & 0 & 0 & 1 & 0 & 0 \\ 1 & 0 & 0 & 0 & 0 & 0 & 1 & 0 \\ 1 & 0 & 0 & 0 & 0 & 0 & 0 & 1 \end{pmatrix}.$$
Note that this matrix does not belong to the set of general matrices studied in \cite{car}. \\

This article is organized as follows: in Section~\ref{s2}, we shall use some elementary properties of unitary divisors to determine a LU-decomposition of the matrix $R_n^*$ and deduce its determinant. In Section~\ref{s3}, following the ideas of \cite{vau}, we shall discuss further on the characteristic polynomial of $R_n^*$ and the algebraic multiplicity of the eigenvalue $1$ of this matrix.

\subsection*{Notation.}

In what follows, $n \geqslant 2$ is a fixed integer and the function $\mu^*$ is the unitary analog of the M\"{o}bius function. We also define
$$M^*(x,n) := \sum_{\substack{k \leqslant x \\ \gcd(k,n) = 1}} \mu^*(k) \quad \left( x > 0, \ n \in \N \right)$$
and simply write $M^*(x) := M^*(x,1)$ which is the unitary analog of the Mertens function. As usual, let $\mathbf{1}(n)=1$ and the \textit{unitary convolution product} of the two arithmetic functions $f$ and $g$ is defined by
$$(f \odot g) (n) = \sum_{d \parallel n} f(d) g(n/d).$$
Finally, from \cite[Theorem~2.5]{coh} it is known that
$$\mu^*(n) = (-1)^{\omega(n)}$$
where $\omega(n)$ is the number of distinct prime factors of $n$, and from \cite[Corollary~2.1.2]{coh} we have the important convolution identity
\begin{alignat}{1}
   \left( \mu^* \odot \mathbf{1} \right) (n) = \begin{cases} 1, & \textrm{if\ } n=1 \\ 0, & \textrm{otherwise}. \end{cases} \label{e1} 
\end{alignat}

\section{The determinant of $R_n^*$}
\label{s2}

We start with the following basic identities involving unitary divisors which will prove to be useful to determine a LU-type decomposition of the matrix $R_n^*$.

\begin{lemma}
\label{le1}
\begin{enumerate}
   \item[]
   \item[{\rm (i)}] Let $i,j$ be positive integers. Then
$$\sum_{\substack{d \parallel j \\ i \parallel j/d}} \mu^*(d) = \begin{cases} 1, & \textrm{if\ } i=j \\ 0, & \textrm{otherwise}. \end{cases}$$
   \item[{\rm (ii)}] Let $1 \leq i \leq n$ be integers. Then
$$\sum_{\substack{k \leq n \\ i \parallel k}} M^* \left( \frac{n}{k},k \right) = 1.$$
\end{enumerate} 
\end{lemma}

\begin{proof}
\begin{enumerate}
   \item[]
   \item[{\rm (i)}] If $i \nparallel j$, then the sum is equal to $0$ since
$$d \parallel j \ \textrm{and\ } i \parallel j/d \Longrightarrow i \parallel j.$$
If $i \parallel j$, then 
$$d \parallel j \ \textrm{and\ } i \parallel j/d \Longleftrightarrow d \parallel j/i$$
so that using \eqref{e1} we get
$$\sum_{\substack{d \parallel j \\ i \parallel j/d}} \mu^*(d) = \sum_{d \parallel j/i} \mu^*(d) = \begin{cases} 1, & \textrm{if\ } j/i=1 \\ 0, & \textrm{otherwise}. \end{cases}$$
   \item[{\rm (ii)}] Using the identity above, we get$$
   1 = \sum_{j \leq n} \sum_{\substack{k \parallel j \\ i \parallel k}} \mu^* \left( \frac{j}{k} \right) = \sum_{\substack{k \leq n \\ i \parallel k}} \sum_{\substack{d \leq n/k \\ \gcd(d,k)=1}} \mu^*(d) = \sum_{\substack{k \leq n \\ i \parallel k}} M^* \left( \frac{n}{k},k \right).$$
\end{enumerate} 
The proof is complete.
\end{proof}

Let $S_n = (s_{ij})$ and $T_n = (t_{ij})$ be the $\left( n \times n \right)$-matrices defined by
$$s_{ij} =  \begin{cases} 1,& \textrm{if\ } i \parallel j \\ 0, & \textrm{otherwise} \end{cases} \quad \textrm{and} \quad t_{ij} = \begin{cases} M^*(n/i,i),& \textrm{if\ }  j=1 \\ 1, & \textrm{if\ } i=j \geq 2 \\ 0, & \textrm{otherwise}. \end{cases}$$
For instance
$$S_8 = \begin{pmatrix} 1 & 1 & 1 & 1 & 1 & 1 & 1 & 1 \\ 0 & 1 & 0 & 0 & 0 & 1 & 0 & 0 \\ 0 & 0 & 1 & 0 & 0 & 1 & 0 & 0 \\ 0 & 0 & 0 & 1 & 0 & 0 & 0 & 0 \\ 0 & 0 & 0 & 0 & 1 & 0 & 0 & 0 \\ 0 & 0 & 0 & 0 & 0 & 1 & 0 & 0 \\ 0 & 0 & 0 & 0 & 0 & 0 & 1 & 0 \\ 0 & 0 & 0 & 0 & 0 & 0 & 0 & 1 \end{pmatrix} \quad \textrm{and} \quad T_8 = \left( \begin{array}{rrrrrrrr} -4 & 0 & 0 & 0 & 0 & 0 & 0 & 0 \\ 0 & 1 & 0 & 0 & 0 & 0 & 0 & 0 \\ 0 & 0 & 1 & 0 & 0 & 0 & 0 & 0 \\ 1 & 0 & 0 & 1 & 0 & 0 & 0 & 0 \\ 1 & 0 & 0 & 0 & 1 & 0 & 0 & 0 \\ 1 & 0 & 0 & 0 & 0 & 1 & 0 & 0 \\ 1 & 0 & 0 & 0 & 0 & 0 & 1 & 0 \\ 1 & 0 & 0 & 0 & 0 & 0 & 0 & 1 \end{array} \right) .$$

We now are in a position to prove the first result concerning the matrix $R_n^*$. 

\begin{theorem}
\label{t2}
Let $n \geq 2$ be an integer. Then $R_n^* = S_n T_n$. In particular
$$\det R_n^* = M^*(n) = \sum_{k=1}^n \mu^*(k).$$
\end{theorem}

\begin{proof}
Set $S_n T_n=(x_{ij})$. If $j=1$, using Lemma~\ref{le1}~(ii) we get
$$x_{i1} = \sum_{k=1}^{n} s_{ik} t_{k1} = \sum_{\substack{ k \leqslant n \\ i \parallel k}} M^* \left ( \frac{n}{k},k \right ) = 1 = \rho_{i1}.$$
If $j \geqslant 2$, then $t_{1j} = 0$ and thus
$$x_{ij} = \sum_{k=2}^{n} s_{ik} t_{kj} = s_{ij} = \begin{cases} 1, & \mathrm{if\ } i \parallel j \\ 0, & \mathrm{otherwise} \end{cases} = \rho_{ij}$$
which is the desired result. The second assertion follows at once from
$$\det R_n^* = \det S_n \det T_n = \det T_n = M^*(n).$$
The proof is complete.
\end{proof}

\begin{corollary}
The Riemann hypothesis is true if and only if, for each $\varepsilon > 0$
$$\det R_n^* = O \left( n^{1/2+\varepsilon} \right).$$
\end{corollary}

\section{The characteristic polynomial of $R_n^*$}
\label{s3}

\subsection{The 'trivial` eigenvalue $1$}

Let $\ell = \left \lfloor \frac{\log n}{\log 2} \right \rfloor$. It is proved in \cite{vau} that $1$ is an eigenvalue of the Redheffer's matrix $R_n$ of algebraic multiplicity equal to $n - \ell - 1$. We will show in this section that the algebraic multiplicity $m_n$ of the eigenvalue $1$ of $R_n^*$ may be somewhat larger.

To this end, we first note that the method developed in \cite{car,vau} to determine the characteristic polynomial of Redheffer type matrices can readily be adapted to the matrix $R_n^*$ which yields
$$\det \left( \lambda I_n - R_n^* \right) = \left( \lambda - 1 \right)^n - (n-1)\left( \lambda - 1 \right)^{n-2} - \sum_{k=2}^{\ell} S^*_k(n) \left( \lambda - 1 \right)^{n-k-1}$$
where 
$$S_k^*(x) = \sum_{m \leqslant x} D_k^*(m)$$
and
$$D_k^*(m) = \sum_{\substack{m=d_1 \dotsb d_k \\ i \neq j \Rightarrow \gcd(d_i,d_j)=1 \\ d_j \geqslant 2}} 1.$$
Note that the arithmetic function $D_k^*$ is the unitary analogue to the strict divisor function $D_k$ which can be found in the coefficients of the characteristic polynomial of $R_n$. Hence, using \cite[(14)]{san} and \cite[(4)]{boy} successively, we get for any $m,k \in \Z_{\geqslant 1}$
$$D_k^*(m) = \sum_{j=0}^k (-1)^{k-j} {k \choose j} \tau_j^* (m) = \sum_{j=0}^k (-1)^{k-j} {k \choose j} j^{\omega(m)} = k! \displaystyle {\omega(m) \brace k}$$
where ${n \brace k}$ is the Stirling number of the second kind. In particular, for any $m ,k\in \Z_{\geqslant 1}$ such that $\omega(m) < k$, we have $D_k^*(m) = 0$. We now are in a position to prove the following result.

\begin{theorem}
\label{eq:th-mn}
Let $n \geqslant 1$. Then the algebraic multiplicity $m_n$ of the eigenvalue $1$ of $R_n^*$ satisfies
$$m_n = n - k_n$$
where the sequence $(k_n)$ of positive integers is given by 
\begin{equation}
   k_1 = 0 \quad \text{and} \quad k_n = \max \left( k_{n-1}, \omega(n)+1 \right) \quad \left( n \in \Z_{ \geqslant 2} \right). \label{eq:kn}
\end{equation}
In particular
$$n  - \left \lfloor \frac{\np{1.3841}\log n}{\log \log n} \right \rfloor -1 \overset{\left( n \geqslant 3 \right)}{\leqslant} m_n \overset{\left( n \geqslant 6 \right) }{\leqslant} n - \left \lfloor \frac{\log n}{\log \log n} \right \rfloor.$$
Also, for any $n \geqslant 3$
$$m_n = n - \frac{\log n}{\log \log n} + O^\star \left( \frac{2 \log n}{(\log \log n)^2} \right).$$
\end{theorem}

\begin{proof}
Since $m_1=1 = 1 - k_1$, we may suppose $n \geqslant 2$. We first show by induction that, for any $n \in \Z_{\geqslant 2}$, there exists a sequence $(k_n)$ of positive integers such that, for any $m \in \{1, \dotsc,n \}$, $\omega(m) < k_n$, this sequence being given by \eqref{eq:kn}. Indeed, the assertion is obviously true for $n=2$ since $k_2=2$, and if we assume it for some $n \geqslant 2$, then, for any $m \in \{1, \dotsc,n+1\}$, either $m \in \{1, \dotsc,n \}$ and then $\omega(m) < k_n$ by induction hypothesis, or $m=n+1$ and $\omega(m) < 1 + \omega(n+1)$, so that, for any $m \in \{1, \dotsc,n+1\}$, we get $\omega(m) < \max \left( k_{n}, \omega(n+1)+1 \right) = k_{n+1}$. We now prove that $k_n$ is the smallest nonnegative integer satisfying this property, i.e. if there exists $h_n \in \Z_{\geqslant 0}$ such that, for all $m \in \{1, \dotsc,n \}$, $\omega(m) < h_n$, then $k_n \leqslant h_n$. Suppose on the contrary that $h_n < k_n = \max \left( k_{n-1}, \omega(n)+1 \right)$. If $h_n < \omega(n)+1$, then $\omega(n) < h_n < \omega(n)+1$ giving a contradiction, and hence $h_n < k_{n-1} = \max \left( k_{n-2}, \omega(n-1)+1 \right)$. Again, if $h_n < \omega(n-1)+1$, then $\omega(n-1) < h_n < \omega(n-1)+1$ which is impossible, and hence $h_n < k_{n-2}$. Continuing this way we finally get $h_n < k_1 = 1$, resulting in a contradiction.
\item[]
Hence for any $m \in \{1, \dotsc,n \}$, we infer that $D_{k}^* (m) = 0$ for any $k \geqslant k_n$, and thus
$$S_{k}^* (n) = 0 \quad \left( k \geqslant k_n \right) \quad \textrm{and} \quad S_{k}^* (n) \ne 0 \quad \left( k < k_n \right)$$
completing the proof of the first part of the theorem. For the second part, we first numerically check the inequality for $n \in \{3, \dotsc , 29 \}$ and assume $n \geqslant 30$, so that $k_n \geqslant 4$. Next, for any $k \in \Z_{\geqslant 1}$, define $N_k := p_1 \dotsb p_k$. It is easy to see that $k_n$ is the unique positive integer such that $N_{k_n-1} \leqslant n < N_{k_n}$ (see also \cite[p. 380]{rob}), so that, from \cite[Theorem~11]{rob}, we derive
$$k_n = 1 + \omega \left( N_{k_n-1} \right) \leqslant 1+ \frac{\np{1.3841}\log N_{k_n-1}}{\log \log N_{k_n-1}} \leqslant 1+ \frac{\np{1.3841}\log n}{\log \log n}.$$
Furthermore, \cite[Theorem~10]{rob} yields
$$k_n = \omega \left( N_{k_n} \right) > \frac{\log N_{k_n}}{\log \log N_{k_n}} > \frac{\log n}{\log \log n}$$
which proves the inequality. We proceed similarly for the last estimate: first check it for $n \in \{3, \dotsc , \np{2309} \}$, then assume $n \geqslant \np{2310}$ so that $k_n \geqslant 6$, and use \cite[Theorem~12]{rob} to get
\begin{align*}
   k_n & \leqslant 1+ \frac{\log N_{k_n-1}}{\log \log N_{k_n-1}} + \frac{\np{1.4575} \log N_{k_n-1}}{\left (\log \log N_{k_n-1} \right )^2} \\
   & < \frac{\log N_{k_n-1}}{\log \log N_{k_n-1}} + \frac{2 \log N_{k_n-1}}{\left (\log \log N_{k_n-1} \right )^2} \\
   & \leqslant \frac{\log n}{\log \log n} + \frac{2 \log n}{\left (\log \log n \right )^2}
\end{align*}
which terminates the proof of Theorem~\ref{eq:th-mn}.
\end{proof}

\subsection{The "dominant" eigenvalues}

We first notice that
\begin{align*}
   S_2^*(x) &= 2 \sum_{m \leqslant x} {\omega(m) \brace 2} = \sum_{m \leqslant x}  2^{\omega(m)} - 2 \lfloor x \rfloor \\
   &= \frac{x \log x}{\zeta(2)} + 2x \left( \gamma - \frac{3}{2} - \frac{\zeta^{\, \prime}}{\zeta} (2) \right) + o \left( x^{1/2} \right).
\end{align*}

Now following the argument leading to \cite[(18)]{vau}, we deduce that $R_n^*$ has two "dominant" eigenvalues $\lambda_{\pm}$ satisfying the following estimate.

\begin{prop}
For all $n \in \Z_{\geqslant 3}$
$$\lambda_{\pm} = \pm \sqrt{n} + \frac{\log n}{2 \zeta(2)} + \gamma - \frac{1}{2} - \frac{\zeta^{\, \prime}}{\zeta} (2) + O \left( n^{-1/2} \log^2 n \right).$$
\end{prop}

\end{document}